\documentclass[12pt,reqno]{amsart}
\usepackage{amssymb,amsmath,amsthm,hyperref}
\newtheorem{theorem}{Theorem}
\newtheorem{lemma}[theorem]{Lemma}
\newtheorem{corollary}[theorem]{Corollary}
\newtheorem{proposition}[theorem]{Proposition}

\theoremstyle{remark}

\theoremstyle{definition}

\newtheorem{definition}[theorem]{Definition}

\numberwithin{theorem}{section} 
\numberwithin{equation}{section}
\numberwithin{example}{section}

\title[A limiting case of a higher-dimensional Kronecker-type identity]{A Kronecker-type identity and the representations of a number as a sum of three squares}

\author{Eric T. Mortenson}

\begin{document}

\date{13 February 2017}



\begin{abstract}
By considering a limiting case of a Kronecker-type identity, we obtain an identity found by both Andrews and Crandall.  We then use the Andrews--Crandall identity to give a new proof of a formula of Gauss for the representations of a number as a sum of three squares.  From the Kronecker-type identity, we also deduce Gauss's theorem that every positive integer is representable as a sum of three triangular numbers.
\end{abstract}

\address{Max-Planck-Institut f\"ur Mathematik, Vivatsgasse 7, 53111 Bonn, Germany}
\email{etmortenson@gmail.com}
\maketitle

\section{Introduction}

 Let $q\in\mathbb{C}$ with $0<|q|<1$ and define $\mathbb{C}^{\star}:=\mathbb{C}-\{0\}$.  Recall the notation
\begin{gather*}
(x)_{\infty}=(x;q)_{\infty}:=\prod_{k= 0}^{\infty}(1-xq^{k}), \ \ (x_1,\dots,x_n;q)_{\infty}:=(x_1;q)_{\infty}\cdots (x_n;q)_{\infty}.
\end{gather*}
The following identity was known to Kronecker \cite{Kron1}, \cite[pp. $309$--$318$]{Kron2}, \cite[pp. $70$--$71$]{Weil}. For  $x,y\in \mathbb{C}^*$ where $|q|<|x|<1$ and $y$ neither zero nor an integral power of $q$
\begin{equation}
\sum_{r\in \mathbb{Z}}\frac{x^r}{1-yq^{r}}=\frac{(q)_{\infty}^2(xy,q/xy;q)_{\infty}}{(x,q/x,y,q/y;q)_{\infty}}.\label{equation:kronecker-original}
\end{equation}
However, Kronecker's identity is also a special case of Ramanujan's ${}_{1}\psi_{1}$-summation.  

If we place the additional restriction $|q|<|y|<1$, we have a more symmetric form
\begin{equation}
\Big ( \sum_{r,s\ge 0}-\sum_{r,s<0} \Big)q^{rs}x^ry^s=\frac{(q)_{\infty}^2(xy,q/xy;q)_{\infty}}{(x,q/x,y,q/y;q)_{\infty}},\label{equation:kronecker}
\end{equation}
which has wide variety of applications.  Limiting cases of (\ref{equation:kronecker}) yield well-known results on the representations of a number as a sum of squares \cite{CL, Warn}.  Define $r_s(n)$ to be the number of representations of $n$ as a sum of $s$ squares and define the generating function
\begin{align}
R_s(q)&:=\sum_{n\ge 0 }r_s(n)(-q)^n
=\Big ( \sum_{m\in \mathbb{Z} }(-1)^mq^{m^2}\Big )^s
=\Big ((q)_{\infty}/(-q)_{\infty} \Big )^s,
\end{align}
where the last equality follows from Jacobi's triple product identity.  Multiply both sides of (\ref{equation:kronecker}) by $(1-x)(1-y)/(1-xy)$ and rewrite to have
\begin{equation}
1+\frac{(1-x)(1-y)}{1-xy}\sum_{r,s\ge 1}q^{rs}(x^ry^s-x^{-r}y^{-s})
=\frac{(xyq)_{\infty}(q/xy)_{\infty}(q)_{\infty}^2}{(xq)_{\infty}(q/x)_{\infty}(yq)_{\infty}(q/y)_{\infty}}.\label{equation:kronecker-alt}
\end{equation}
Taking the limit $x=y\rightarrow -1$ of (\ref{equation:kronecker-alt}) yields Jacobi's four-square theorem
\begin{equation}
r_4(n)=8\sum_{\substack{4\nmid d, d \mid n}}d,\label{equation:jacobi}
\end{equation}
which implies Lagrange's theorem that every positive integer is a sum of four squares.  Taking the limit $x$, $y^2\rightarrow -1$ of (\ref{equation:kronecker-alt}) gives a result of Gauss and Lagrange
\begin{equation}
r_2(n)=4(d_1(n)-d_3(n)),\label{equation:fermat}
\end{equation}
where $d_k(n)$ is the number of divisors of $n$ congruent to $k$ modulo $n$.  Identity (\ref{equation:fermat}) implies Fermat's two-square theorem.

Do higher-dimensional generalizations of (\ref{equation:kronecker-original}) yield results on sums of squares?  In recent work, we found that Kronecker's identity (\ref{equation:kronecker-original}) has a double-sum analog:

\begin{theorem}\cite[Theorem $1.1$]{M1}\label{theorem:result} For  $x,y,z\in \mathbb{C}^*$ where $|q|<|y|<1$,  $|q|<|z|<1$, and $x$ neither zero nor an integral power of $q$,
{\allowdisplaybreaks \begin{align}
\Big ( \sum_{s,t \ge 0}&-\sum_{s,t<0} \Big)\frac{q^{st}y^sz^t}{1-xq^{s+t}} \label{equation:thm-result}\\
&=   \frac{(xy,q^2/xy;q^2)_{\infty}}{(x,y,q/x,q/y;q)_{\infty}}
 \frac{(q;q)_{\infty}^2}{(q^2;q^2)_{\infty}}\sum_{k\in \mathbb{Z}} \frac{(-1)^kq^{k^2}(xy)^k}{1+q^{2k}z}+\textup{idem}(z;x,y)\notag \\
&\ \ \ \ \ -2  \frac{(q^2;q^2)_{\infty}^3}{(x,y,z,q/x,q/y,q/z;q)_{\infty}} \frac{(xy,xz,yz,q^2/xy,q^2/xz,q^2/yz;q^2)_{\infty}}{(-x,-y,-z,-q^2/x,-q^2/y,-q^2/z;q^2)_{\infty}},\notag
\end{align}}%
where ``\textup{idem}$(x;x_2,\dots,x_n)$'' means that the previous expression is repeated with $x$ and $x_i$ interchanged for each $i$.
\end{theorem}

A limiting case of Theorem \ref{theorem:result} gives a new proof of an identity found by Crandall in his work on the Madelung constant \cite[$(6.2)$]{Cr}:
\begin{theorem} \label{theorem:ac-identity} For positive integers $n$, we have
\begin{equation}
r_3(n)=6(-1)^{n+1}\sum_{\substack{r,s \ge 1\\rs=n}}(-1)^{r+s} +4(-1)^{n+1}\sum_{\substack{r,s,t \ge 1\\rs+rt+st=n}}(-1)^{r+s+t}.\label{equation:gauss-alt}
\end{equation}
\end{theorem}
\noindent Crandall states that (\ref{equation:gauss-alt}) may also be obtained via Andrews's identity \cite[$(5.16)$]{A}:
\begin{equation}
\sum_{n=0}^{\infty}r_3(n)(-q)^n=1+4\sum_{n=1}^{\infty}\frac{(-1)^nq^n}{1+q^n}-
2\sum_{\substack{n=1\\|j|<n}}^{\infty}\frac{q^{n^2-j^2}(1-q^n)(-1)^j}{1+q^n}.
\end{equation}

Define $N_3(n)$ to be the number of integral solutions to $n=x^2+y^2+z^2$, where $\gcd(x,y,z)=1$.  Gauss showed \cite[Art. $291$, $292$]{G}:
\begin{theorem}[Gauss] \label{theorem:gauss} Let $\delta_n=1$ except for $\delta_1=1/2$ and $\delta_3=1/3$.  We have, 
\begin{align*}
N_3(n)&=12\delta_n h(-4n), \ \textup{for }n\equiv 1,\ 2,\ 5, \textup{ or } 6 \pmod 8,\\
N_3(n)&=24\delta_n h(-n), \  \textup{for }n\equiv 3 \pmod 8,
\end{align*}
where $h(d)$ is the order of the class group of discriminant $d$.
\end{theorem}
\noindent A consequence of Gauss's theorem is the celebrated local to global principle:

\smallskip
\noindent {\em Legendre/Gauss $(1800)$: $r_3(n)>0$ if and only if $n\ne 4^a(8b+7)$.}  

\smallskip
Using the Andrews--Crandall identity (\ref{equation:gauss-alt}), we give a new proof of the formula:

\begin{theorem}[Gauss] \label{theorem:gauss-withsquare} We have
\begin{align*}
r_3(n)&=12H(4n), \ \textup{for }n\equiv 1,\ 2,\ 5, \textup{ or } 6 \pmod 8,\\
r_3(n)&=24H(n), \  \textup{for }n\equiv 3 \pmod 8,\\
r_3(n)&=0, \  \textup{for }n\equiv 7 \pmod 8,\\
r_3(n)&=r_3(n/4),  \  \textup{for }n\equiv 0 \pmod 4,
\end{align*}
where $H(d)$ is the Hurwitz class number of discriminant $d$.
\end{theorem}

Many authors have given treatments of Theorem \ref{theorem:gauss-withsquare}. Hecke \cite{He} discussed Theorem \ref{theorem:gauss-withsquare} with reference to Kronecker \cite[p. 253]{Kron0}.  Kronecker proved Theorem \ref{theorem:gauss-withsquare} using identities that he states follow from the theory of elliptic forms \cite[p. 108]{Di}, \cite[pp. 321--322]{Sm}.  Hermite expanded upon Kronecker's identities \cite[Chapter VI, p. 92]{Di}, \cite[pp. 321 ff., pp. 338]{Sm}, and the method of Hermite was translated by Liouville into a purely arithmetical deduction of Kronecker's relations \cite[Chapter XIII]{UH}.  One also has Weil's \cite{Weil2}.

Modularity can also be used to prove identities on sums of squares.  Identity (\ref{equation:jacobi}) is a classic example used to motivate the use of modular forms \cite{DS}.    Motivated by a suggestion of Hecke \cite{He}, Hirzebruch and Zagier \cite{HZ} completed
\begin{equation}
\sum_{n=0}^{\infty}H(n)q^n
\end{equation}
to a function which transforms under $\Gamma_0(4)$ like a modular form of weight $\tfrac{3}{2}$:
\begin{theorem}\label{theorem:HZ-Thm2}\cite[Theorem $2$, p. 92]{Za} For $z\in\mathfrak{H}$, we have
\begin{equation*}
\mathcal{F}(z)=\sum_{n=0}^{\infty}H(n)q^n+y^{-1/2}\sum_{f=-\infty}^{\infty}\beta(4\pi f^2y)q^{-f^2},
\end{equation*}
where $y=\textup{Im}(z)$, $q=e^{2\pi i z}$ and $\beta(x)$ is defined by
\begin{equation*}
\beta(x)=\frac{1}{16\pi} \int_{1}^{\infty}u^{-3/2}e^{-xu}du, \ \ (x\ge0).
\end{equation*}
\end{theorem} 
\noindent As a corollary to Theorem \ref{theorem:HZ-Thm2}, they prove Theorem \ref{theorem:gauss-withsquare} \cite[pp. 92--93]{HZ}.

Of course, there is another well-known result that can be deduced from Theorem \ref{theorem:result}.  Let us define $r_{3\Delta}(n)$ to be the number of representations of $n$ as a sum of three triangular numbers, where a triangular number is a number of the form $k(k-1)/2$.  We write the generating function as
\begin{equation}
\sum_{n=0}^{\infty}r_{3\Delta}(n)q^n:=\Big ( \sum_{n=0}^{\infty}q^{\binom{n+1}{2}}\Big )^3=\Big ( \frac{(q^2;q^2)_{\infty}^2}{(q;q)_{\infty}}\Big )^3, 
\end{equation}
where we have used Jacobi's triple product identity.  Gauss discovered that every positive integer is representable as a sum of three triangular numbers, \cite[Art. $293$]{G}.  In his diary, he made the entry:

\smallskip
\noindent {\em Gauss (1796): EYPHKA! num = $\Delta+\Delta+\Delta$.}

\smallskip
A specialization of Theorem \ref{theorem:result} yields the following identity, which gives as an immediate corollary Gauss's result on triangular numbers. 
\begin{theorem}\label{theorem:EYPHKA} We have
\begin{equation}
\sum_{n=0}^{\infty}r_{3\Delta}(n)q^n=1+3\sum_{r\ge 1}^{\infty}q^r+3\sum_{r,s \ge 1}q^{2rs+r+s} 
+\Big ( \sum_{r,s,t > 0}+\sum_{r,s,t < 0}\Big )q^{2rs+2rt+2st+r+s+t}.\label{equation:id-EYHPKA}
\end{equation}
\end{theorem}
\noindent Gauss's EYPHKA theorem is also a corollary of an identity of Andrews \cite[$(5.17)$]{A}.

In Section \ref{section:notation}, we recall basic facts about quadratic forms and class numbers.  In Sections \ref{section:Gauss} and \ref{section:EYPHKA}, we demonstrate how Theorem \ref{theorem:result} yields (\ref{equation:gauss-alt}) and (\ref{equation:id-EYHPKA}) respectively.  In Section \ref{section:idea}, we sketch the idea behind our proof of Theorem \ref{theorem:gauss-withsquare}.  In Section \ref{section:AC-rework}, we rewrite (\ref{equation:gauss-alt}).  In Section \ref{section:gauss-withsquare}, we prove Theorem \ref{theorem:gauss-withsquare}.   In Section \ref{section:remarks}, we make our concluding remarks.

\section{Preliminaries}\label{section:notation}

We review facts on binary quadratic forms and class numbers \cite{Co, Za}.  Let
\begin{equation}
f(x,y):=ax^2+bxy+cy^2
\end{equation}
be a binary quadratic form.  The discriminant of $f(x,y)$ is defined $D(f):=b^2-4ac$.

\begin{definition} \cite[Definition $5.3.2$]{Co}
 A positive definite quadratic form $(a,b,c)$ of discriminant $D$ is said to be reduced if $|b|\le a\le c$ and if, in addition, when one of the two inequalities is an equality (i.e. either $|b|=a$ or $a=c$), then $b\ge0$.
\end{definition}
A reduced form is said to be primitive if $\gcd(a,b,c)=1$ and imprimitive otherwise.  The number of reduced forms of a given discriminant $D$ is finite in number \cite[p. 59]{Za}.
\begin{definition}\cite[p. 226]{Co}
We define the class number $h(D)$ to be the number of primitive positive definite reduced quadratic forms of discriminant $D$.
\end{definition}

\begin{definition}\label{definition:Cohen-definition536}  \cite[Definition $5.3.6$]{Co} Let $N$ be a non-negative integer.  The Hurwitz class number $H(N)$ is defined as follows.
\begin{itemize}
\item[(1)]  If $N\equiv 1, 2 \pmod 4$ then $H(N)=0$.
\item[(2)]  If $N=0$ then $H(0)=-1/12$.
\item[(3)]  Otherwise (i.e. if $N\equiv 0 \ \textup{or } 3 \pmod 4$ and $N>0$) we define $H(N)$ as the class number of not necessarily primitive (positive definite) quadratic forms of discriminant $-N$, except that forms equivalent to $a(x^2+y^2)$ should be counted with weight $1/2$, and those equivalent to $a(x^2+xy+y^2)$ with weight $1/3$.
\end{itemize}
\end{definition}

Let $\omega(D)$ be the number of roots of unity in the quadratic order of discriminant $D$, then $\omega(-3)=6$, $\omega(-4)=4$, and $\omega(D)=2$ for $D<-4$.  Write $D=D_0f^2$ where $D_0$ is a fundamental discriminant, i.e. $D\ne 1$ and either $D\equiv 1 \pmod 4$ and is sqaurefree, or $D\equiv 0 \pmod 4$, $D/4$ is squarefree and $D/4\equiv 2$ or $3$ $\pmod4$.  We recall a consequence of Dirichlet's \cite[Proposition $5.3.12$]{Co}, \cite[p. $72$]{Za},  see also \cite[p. $228$]{Co}, \cite[pp. $74$, $95$]{Za}:
\begin{equation}
\frac{h(D)}{\omega(D)}=\frac{h(D_0)}{\omega(D_0)}f\prod_{p\mid f}\Big ( 1-\frac{\big (\tfrac{D_0}{p}\big )}{p}\Big),\label{equation:ratio-p228}
\end{equation}
where $\big (\tfrac{\bullet}{p}\big )$ is the Kronecker--Jacobi symbol \cite[Definition $1.4.8$]{Co}.  One could also obtain (\ref{equation:ratio-p228}) via a purely algebraic proof outlined in a series of exercises \cite[Ch. $2.7$, Probs. $6$--$11$]{BoSa}, see also \cite[Art.  $113$, pp. 246--251]{Sm}.

\begin{lemma} \cite[Lemma $5.3.7$]{Co} \label{lemma:Co-lemma537} Define $h^{\prime}(D):=h(D)/(\omega(D)/2)$.  For $N>0$ we have
\begin{equation}
H(N)=\sum_{d^2\mid N}h^{\prime}(-N/d^2),
\end{equation}
and in particular if $-N$ is a fundamental discriminant, we have $H(N)=h(-N)$ except in the special cases $N=3$ ($H(3)=1/3$ and $h(-3)=1$) and $N=4$ ($H(4)=1/2$ and $h(-4)=1$).
\end{lemma}

\begin{lemma} \label{lemma:mult-4} For positive numbers $n$, where $n\equiv 3 \pmod 4$, we have
{\allowdisplaybreaks \begin{align}
H(4n)=4H(n) \ \textup{for }n\equiv 3\pmod 8,\label{equation:delta-3mod8}\\
H(4n)=2H(n) \ \textup{for }n\equiv 7\pmod 8.\label{equation:delta-7mod8}
\end{align}}%
\end{lemma}
\begin{proof} By Lemma \ref{lemma:Co-lemma537},
{\allowdisplaybreaks \begin{align*}
H(4n)&=\sum_{d^2\mid (4n)}h^{\prime}(-4n/d^2)\\
&=\sum_{2\nmid d, d^2\mid (4n)}h^{\prime}(-4n/d^2)+\sum_{ (2d)^2\mid (4n)}h^{\prime}(-4n/(2d)^2)\\
&=\sum_{d^2\mid n}h^{\prime}(-4n/d^2)+\sum_{ d^2\mid n}h^{\prime}(-n/d^2)\\
&=\sum_{d^2\mid n}h^{\prime}(-\Delta_d)2f_d\prod_{p\mid 2f_d}\Big ( 1-\frac{\big (\tfrac{\Delta_d}{p}\big )}{p}\Big)
 +\sum_{ d^2\mid n}h^{\prime}(-\Delta_d)f_d\prod_{p\mid f_d}\Big ( 1-\frac{\big (\tfrac{\Delta_d}{p}\big )}{p}\Big),
\end{align*}}%
where we used (\ref{equation:ratio-p228}) with $-N/d^2=f_d^2\Delta_d$ and $\Delta_d$ a fundamental discriminant.  For $n\equiv 3 \pmod 8$, we note that  $\big ( \tfrac{\Delta_d}{2} \big )=-1$ and (\ref{equation:delta-3mod8}) follows.  For $n\equiv 7 \pmod 8$ we have $\big ( \tfrac{\Delta_d}{2} \big )=1$ and (\ref{equation:delta-7mod8}) follows.
\end{proof}

\section{Proof of Theorem \ref{theorem:gauss-withsquare}}\label{section:Gauss}

For our purpose we impose the additional restriction $|q|<|x|<1$.  The right-hand side of (\ref{equation:thm-result}) remains the same; however, the left-hand side of (\ref{equation:thm-result}) becomes the symmetric
\begin{equation}
\Big ( \sum_{s,t \ge 0}-\sum_{s,t<0} \Big)\frac{q^{st}y^sz^t}{1-xq^{s+t}}=\Big ( \sum_{r,s,t \ge 0}+\sum_{r,s,t < 0}\Big )q^{st+rs+rt}x^ry^sz^t.\label{equation:LHS-symmetric}
\end{equation}

We set $x=y=z$ in the right-hand side of (\ref{equation:thm-result}) and then take the limit $x\rightarrow -1$.  The theta-quotient in the right-hand side becomes
{\allowdisplaybreaks \begin{align*}
\lim_{x\rightarrow -1}& -2 \frac{(q^2;q^2)_{\infty}^3}{(xq,q/x;q)_{\infty}^3} \frac{(x^2q^2,q^2/x^2;q^2)_{\infty}^3}
 {(-xq^2,-q^2/x;q^2)_{\infty}^3}
  = -2 \frac{(q^2;q^2)_{\infty}^3}{(-q;q)_{\infty}^6} 
=-2R_3(q).
\end{align*}}%
Let us consider an Appell--Lerch function term from the right-hand side.  We have
{\allowdisplaybreaks \begin{align*}
\lim_{x\rightarrow -1}& \frac{(1-x^2)(x^2q^2,q^2/x^2;q^2)_{\infty}}{(1-x)^2(xq,q/x;q)_{\infty}^2}
\frac{(q;q)_{\infty}^2}{(q^2;q^2)_{\infty}} \sum_{k\in \mathbb{Z}}\frac{(-1)^kq^{k^2}(x^2)^k}{1+q^{2k}x}\\
&= \lim_{x\rightarrow -1}\frac{(x^2q^2,q^2/x^2;q^2)_{\infty}}{(1-x)(xq,q/x;q)_{\infty}^2}
\frac{(q;q)_{\infty}^2}{(q^2;q^2)_{\infty}} 
=  \frac{1}{2}\frac{(q^2;q^2)_{\infty}^2}{(-q;q)_{\infty}^4}\frac{(q;q)_{\infty}^2}{(q^2;q^2)_{\infty}}  
=\frac{1}{2}R_3(q),
\end{align*}}%
where we have picked up a non-zero term when $k=0$.    So for $x=y=z\rightarrow -1$, the right-hand side of (\ref{equation:thm-result}) becomes
\begin{equation}
\lim _{x\rightarrow -1} RHS= \frac{3}{2}R_3(q)-2R_3(q)=-\frac{1}{2}R_3(q).\label{equation:new-RHS}
\end{equation}
Let us consider the modified left-hand side of (\ref{equation:thm-result}) which is now (\ref{equation:LHS-symmetric}).  We set $x=y=z$ and then take the limit $x\rightarrow -1$.  We have
{\allowdisplaybreaks \begin{align}
\lim_{x\rightarrow -1 }&\Big ( \sum_{r,s,t \ge 0}+\sum_{r,s,t < 0}\Big )q^{st+rs+rt}x^{r+s+t}\notag \\
&=\lim_{x\rightarrow -1 }\Big [3\sum_{r,s \ge 1}q^{rs}x^{r+s} +3\sum_{r \ge 1}x^r+1
+\Big ( \sum_{r,s,t > 0}+\sum_{r,s,t < 0}\Big )q^{st+rs+rt}x^{r+s+t}\Big ] \notag \\
&=\lim_{x\rightarrow -1 }\Big [3\sum_{r,s \ge 1}q^{rs}x^{r+s} +\frac{3x}{1-x}+1
+\Big ( \sum_{r,s,t > 0}+\sum_{r,s,t < 0}\Big )q^{st+rs+rt}x^{r+s+t}\Big ]\notag \\
&=3\sum_{r,s \ge 1}q^{rs}(-1)^{r+s}-\frac{1}{2}+2\sum_{r,s,t \ge 1}q^{rs+rt+st}(-1)^{r+s+t}.\label{equation:new-LHS}
\end{align}}%
Equating (\ref{equation:new-RHS}) and (\ref{equation:new-LHS}) yields
\begin{equation}
R_3(q)=1+6\sum_{r,s \ge 1}(-q)^{rs}(-1)^{rs+r+s+1} +4\sum_{r,s,t \ge 1}(-q)^{rs+rt+st}(-1)^{rs+rt+st+r+s+t+1}.\label{equation:gauss-gen}
\end{equation}

\section{Proof of Theorem \ref{theorem:EYPHKA}}\label{section:EYPHKA}

As in Section \ref{section:Gauss}, we impose the additional restriction $|q|<|x|<1$.   We set $x=y=z$ in the right-hand side of (\ref{equation:thm-result}) and then make the subsititutions: $q\mapsto q^2$, $x\mapsto q$.   The theta quotient in the right-hand side becomes
\begin{equation*}
 -2  \frac{(q^2;q^2)_{\infty}^3}{(xq,q/x;q)_{\infty}^3}  
 \frac{(x^2q^2,q^2/x^2;q^2)_{\infty}^3}{(-xq^2,-q^2/x;q^2)_{\infty}^3} 
\mapsto -2 \frac{(q^4;q^4)_{\infty}^3}{(q^{3},q;q^2)_{\infty}^3} 
 \frac{(q^6,q^2;q^4)_{\infty}^3}{(-q^{5},-q^{3};q^4)_{\infty}^3}
= -2\frac{(q^2;q^2)_{\infty}^6}{(q;q)_{\infty}^3}.
\end{equation*}
For the Appell--Lerch function terms, we have
\begin{align*}
 \frac{1+x}{1-x}&\frac{(x^2q^2,q^2/x^2;q^2)_{\infty}}{(xq,q/x;q)_{\infty}^2}
\frac{(q;q)_{\infty}^2}{(q^2;q^2)_{\infty}} \sum_{k\in \mathbb{Z}}\frac{(-1)^kq^{k^2}(x^2)^k}{1+q^{2k}x}\\
&\mapsto  \frac{1+q}{1-q}\frac{(q^6,q^2;q^4)_{\infty}}{(q^3,q;q^2)_{\infty}^2}
\frac{(q^2;q^2)_{\infty}^2}{(q^4;q^4)_{\infty}} \sum_{k\in \mathbb{Z}}\frac{(-1)^kq^{4\binom{k+1}{2}}}{1+q^{4k}q}
=\frac{(q^2;q^2)_{\infty}^6}{(q;q)_{\infty}^3},
\end{align*}
where the equality follows from using the classical partial fraction expansion for the reciprocal of Jacobi's theta product \cite[p. 136]{TM}:
\begin{equation*}
\sum_{n\in\mathbb{Z}}\frac{(-1)^nq^{\binom{n+1}{2}}}{1-q^{n}z}=\frac{(q;q)_{\infty}^2}{(z,q/z;q)_{\infty}}, \ \textup{where} \ z\not =q^n, n\in \mathbb{Z}.
\end{equation*}
So the right-hand side of (\ref{equation:thm-result}) becomes 
\begin{equation}
(q^2;q^2)_{\infty}^6/(q;q)_{\infty}^3. \label{equation:new-RHS-EYPHKA}
\end{equation}

Let us consider the modified left-hand side of (\ref{equation:thm-result}) which is now (\ref{equation:LHS-symmetric}).  We set $x=y=z$ and then make the subsititutions: $q\mapsto q^2$, $x\mapsto q$.  We have
 {\allowdisplaybreaks \begin{align}
\Big ( \sum_{r,s,t \ge 0}&+\sum_{r,s,t < 0}\Big )q^{st+rs+rt}x^ry^sz^t\notag \\
&=\Big [3\sum_{r,s \ge 1}q^{rs}x^{r+s} +3\sum_{r \ge 1}x^r+1
+\Big ( \sum_{r,s,t > 0}+\sum_{r,s,t < 0}\Big )q^{st+rs+rt}x^{r+s+t}\Big ] \notag\\
&\ \ \ \ \ \mapsto 1+3\sum_{r\ge 1}^{\infty}q^r+3\sum_{r,s \ge 1}q^{2rs+r+s} 
+\Big ( \sum_{r,s,t > 0}+\sum_{r,s,t < 0}\Big )q^{2rs+2rt+2st+r+s+t}.\label{equation:new-LHS-EYPHKA}
\end{align}}%
Equating (\ref{equation:new-RHS-EYPHKA}) and (\ref{equation:new-LHS-EYPHKA}) yields Theorem \ref{theorem:EYPHKA}.

\section{The Idea behind the proof}\label{section:idea}
Note that we have the decomposition \cite[p. 377]{Cr}:
\begin{equation}
\sum_{\substack{r,s,t\ge1\\rs+rt+st=n}}1=6\sum_{\substack{r>s>t>0\\rs+rt+st=n}}1
+3\sum_{\substack{r,t\ge1\\r\ne t \\r^2+2rt=n}}1+\sum_{\substack{r\ge1\\3r^2=n}}1.\label{equation:AC-decompose}
\end{equation}
Hence we can rewrite (\ref{equation:gauss-alt})
{\allowdisplaybreaks \begin{align}
r_3(n)&=12(-1)^{n+1}\Big [ \frac{1}{2}\sum_{\substack{r,s \ge 1\\rs=n}}(-1)^{r+s} \notag\\
&\ \ \ \ \ +2\sum_{\substack{r>s>t>0\\rs+rt+st=n}}(-1)^{r+s+t}
+\sum_{\substack{r,t\ge1\\r\ne t \\r^2+2rt=n}}(-1)^{t}+\frac{1}{3}\sum_{\substack{r\ge1\\3r^2=n}}(-1)^{r} \Big ].\label{equation:AC-qf-weights}
\end{align}}%

We establish a correspondence between solutions $rs+rt+st=n$ and reduced binary quadratic forms $ax^2+bxy+cy^2$ of discriminant $-4n$.  We see
\begin{align*}
(b/2)(a-(b/2))+(b/2)(c-(b/2))+(a-(b/2))(c-(b/2))=ac-(b/2)^2=n
\end{align*}
is equivalent to
\begin{equation*}
b^2-4ac=-4n.
\end{equation*}
Hence a solution
\begin{equation*}
rs+rt+st=n
\end{equation*}
with $0< t\le s \le r$, corresponds to the reduced quadratic form
\begin{equation*}
(s+t)x^2+(2t)xy+(r+t)z^2
\end{equation*}
with discriminant
\begin{equation*}
(2t)^2-4(s+t)(t+r)=4t^2-4(t^2+rs+rt+st)=-4(rs+rt+st)=-4n.
\end{equation*}

We list the types of primitive and imprimitive reduced forms $ax^2+bxy+cy^2$: 
\begin{itemize}
\item[(I)] forms $(a,b,c)$ where $0=b<a\le c$,
\item[(II)] forms $(a,b,c)$ where $0<|b|<a<c$, i.e. a pair $(a,|b|,c)$ and $(a,-|b|,c)$,
\item [(III)] forms $(a,b,c)$ where $0<b=a<c$ or $0<b<a=c$,
\item[(IV)]  forms $(a,b,c)$ where $0<b=a=c$.
\end{itemize}

The right-hand side of (\ref{equation:AC-qf-weights}) counts precisely the primitive and imprimitive reduced forms of discriminant $-4n$; however, the reduced forms have associated weights.  The first sum counts forms of type (I), the second sum counts forms of type (II), the third sum counts forms of type (III), and the fourth sum counts forms of type (IV).  The proof of Theorem \ref{theorem:gauss-withsquare} boils down to understanding the weights.

\section{Reworking the Andrews--Crandall identity}\label{section:AC-rework}
The goal of this section is to prove the following two propositions.
\begin{proposition}\label{proposition:1mod4id} For $n=4m+1$, we have
\begin{equation}
r_3(n)=6\sum_{\substack{d\ge 1 \\ d \mid n}}1+24\sum_{\substack{r>s>t>0\\rs+rt+st=n}}1
+12\sum_{\substack{r,t\ge1\\r\ne t \\r^2+2rt=n}}1.
\end{equation}
\end{proposition}
\begin{proposition}\label{proposition:2mod4id}  For $n=4m+2$, we have
\begin{equation}
r_3(n)=12\sum_{\substack{d\ge 1 \\ d \mid (n/2)}}1+24\sum_{\substack{r>s>t>0\\rs+rt+st=n}}1.
\end{equation}
\end{proposition}

We recall that Theorem \ref{theorem:ac-identity} has the following two natural corollaries \cite[p. 376]{Cr}:
\begin{corollary} \label{corollary:ac-n-odd} Suppose $n$ is odd, then 
\begin{equation}
r_3(n)=6\sum_{\substack{d\ge 1 \\ d \mid n}}1+4\sum_{\substack{r,s,t \ge 1\\rs+rt+st=n}}(-1)^{r+s+t}.
\end{equation}
\end{corollary}
\begin{corollary} \label{corollary:ac-n-even}Suppose $n$ is even and $k$ is maximal such that $2^k\mid n$, then 
\begin{equation}
r_3(n)=6(3-k)\sum_{\substack{d\ge 1 \\ d \mid (n/2^k)}}1-4\sum_{\substack{r,s,t \ge 1\\rs+rt+st=n}}(-1)^{r+s+t}.
\end{equation}
\end{corollary}

We prove the following lemma.
\begin{lemma}\label{lemma:key-lemma}For $n=4m+1$ and $n=4m+2$, we have
\begin{equation}
(-1)^{n+1}\sum_{\substack{r,s,t\ge1\\rs+rt+st=n}}(-1)^{r+s+t}=\sum_{\substack{r,s,t\ge1\\rs+rt+st=n}}1.\label{equation:sign-eq}
\end{equation}
\end{lemma}
\begin{proof}The result follows from considering the four cases for parity.  If all three of $(r,s,t)$ are odd, then upon rewriting the $3$-tuple according to parity we have
\begin{align*}
(2r+1)(2s+1)+(2r+1)(2t+1)+&(2s+1)(2t+1)\\
&=4(rs+rt+st)+4(r+s+t)+3.
\end{align*}
If we have one even and two odd, we have for example 
\begin{align*}
(2r)(2s+1)+(2r)(2t+1)+&(2s+1)(2t+1)=4(rs+rt+st)+4r+2s+2t+1.
\end{align*}
If we have two even and one odd, we have for example 
\begin{align*}
(2r)(2s)+(2r)(2t+1)+(2s)(2t+1)=4(rs+rt+st)+2r+2s.
\end{align*}
If we have three even, then 
\begin{equation*}
(2r)(2s)+(2r)(2t)+(2s)(2t)=4sr+4rt+4st.
\end{equation*}
If $n=4m+1$, then we can only have one even and two odd. If $n=4m+2$, then we can only have two even and one odd.
\end{proof}

\begin{proof}[Proofs of Propositions \ref{proposition:1mod4id} and \ref{proposition:2mod4id}]

We consider $n=4m+1$.  By parity considerations, we can only have one even and two odd in $(r,s,t)$, so the third sum in (\ref{equation:AC-decompose}) does not occur.  By Corollary \ref{corollary:ac-n-odd} and Lemma \ref{lemma:key-lemma}, we have
\begin{align*}
r_3(n)&=6\sum_{\substack{d\ge 1 \\ d \mid n}}1+4\sum_{\substack{r,s,t \ge 1\\rs+rt+st=n}}1
=6\sum_{\substack{d\ge 1 \\ d \mid n}}1+4\Big ( 6\sum_{\substack{r>s>t>0\\rs+rt+st=n}}1
+3\sum_{\substack{r,t\ge1\\r\ne t \\r^2+2rt=n}}1\Big ). 
\end{align*}

We consider $n=4m+2$.  By parity considerations, we can only have two even and one odd in $(r,s,t)$, so the third sum in (\ref{equation:AC-decompose}) cannot occur.   For the second sum in (\ref{equation:AC-decompose})
\begin{equation*}
(2r)^2+2(2r)(2t+1)=4r^2+8rt+4r\not = 4m+2.
\end{equation*}
So only the first sum in (\ref{equation:AC-decompose}) can occur.  By Corollary \ref{corollary:ac-n-even} and Lemma \ref{lemma:key-lemma}, we have
{\allowdisplaybreaks \begin{align*}
r_3(n)&=12\sum_{\substack{d\ge 1 \\ d \mid (n/2)}}1+4\sum_{\substack{r,s,t \ge 1\\rs+rt+st=n}}1
=12\sum_{\substack{d\ge 1 \\ d \mid (n/2)}}1+4\Big ( 6\sum_{\substack{r>s>t>0\\rs+rt+st=n}}1\Big ). \qedhere
\end{align*}}%
\end{proof}

\section{Proof of Theorem \ref{theorem:gauss-withsquare}}\label{section:gauss-withsquare}

We consider the case $n\equiv 1 \pmod 4$.  We have by Proposition \ref{proposition:1mod4id}:
\begin{equation*}
r_3(n)=12\Big [ \frac{1}{2}\sum_{\substack{d\ge 1 \\ d \mid n}}1+2\sum_{\substack{r>s>t>0\\rs+rt+st=n}}1
+\sum_{\substack{r,t\ge1\\r\ne t \\r^2+2rt=n}}1\Big ].
\end{equation*}

Let us consider the contributions.  The sum
\begin{equation*}
\frac{1}{2}\sum_{\substack{d\ge 1 \\ d \mid n}}1
\end{equation*}
counts the reduced binary quadratic forms of type (I).  They may be primitive or imprimitive.  The factor $1/2$ is used so that for $a\ne c$, the form $ax^2+cy^2$ is only counted once.  If $n$ is a square, say $a^2=n$, then the form $a(x^2+y^2)$ is counted with weight $1/2$, but this is exactly the weight factor from the definition of the Hurwitz class number.

The primitive and imprimitive reduced forms of type (II) are counted by the sum
\begin{equation*}
\sum_{\substack{r>s>t>0\\rs+rt+st=n}}1.
\end{equation*}
To count all the forms $(a,|b|,c)$ and $(a,-|b|,c)$, we introduce a factor of two:  
\begin{equation*}
2\sum_{\substack{r>s>t>0\\rs+rt+st=n}}1.
\end{equation*}

Primitive and imprimitive reduced forms of type (III) are counted by
\begin{equation*}
\sum_{\substack{r,t\ge1\\r\ne t \\r^2+2rt=n}}1.
\end{equation*}

In a solution tuple $(r,s,t)$ where $rs+rt+st=n$ and $n\equiv 1 \pmod 4$, we can only have one even and two odd.  This means that we cannot have a solution $3r^2=n$, so the fourth sum in (\ref{equation:AC-qf-weights}) cannot appear and is therefore not included.  Solutions of the form $3r^2=n$ correspond to the reduced binary quadratic forms of type (IV):
\begin{equation*}
2r(x^2+xy+y^2).
\end{equation*}
Hence there is no $1/3$ contribution to the Hurwitz number.  Thus we have
{\allowdisplaybreaks \begin{align*}
r_3(n)&=12H(4n).
\end{align*}}%

We consider the case $n\equiv 2 \pmod 4$.  We have by Proposition \ref{proposition:2mod4id}:
\begin{equation*}
r_3(n)=12\Big [ 2\Big ( \frac{1}{2}\sum_{\substack{d\ge 1 \\ d \mid (n/2)}}1\Big ) +2\sum_{\substack{r>s>t>0\\rs+rt+st=n}}1\Big ].
\end{equation*}
Solutions $(r,s,t)$ with $rs+rt+st=n$ can only have two even and one odd element.  This precludes solutions of the form $3r^2=n$, so we count no terms with weight $1/3$.  By parity considerations, we have no reduced forms of type (III).  Because $n$ is not a square, we cannot have a reduced form of type $a(x^2+y^2)$.  Arguing as in the previous case yields
\begin{equation*}
r_3(n)=12H(4n).
\end{equation*}

We consider the case $n\equiv 3 \pmod 8$.  From (\ref{equation:AC-qf-weights}) and Corollary \ref{corollary:ac-n-odd}, we have
\begin{equation}
r_3(n)=12\Big [ \frac{1}{2}\sum_{\substack{d\ge 1 \\ d \mid n}}1+2\sum_{\substack{r>s>t>0\\rs+rt+st=n}}(-1)^{r+s+t}
+\sum_{\substack{r,t\ge1\\r\ne t \\r^2+2rt=n}}(-1)^{t}+\frac{1}{3}\sum_{\substack{r\ge1\\3r^2=n}}(-1)^{r}\Big ].\label{equation:n3mod8-master}
\end{equation}
Solution tuples $(r,s,t)$ to $rs+rt+st=n$ have either all three odd or one even and two odd.   Let us consider the contributions.   

Primitive and imprimitive reduced forms of type (I) are counted by
\begin{equation*}
\frac{1}{2}\sum_{\substack{d\ge 1 \\ d \mid n}}1.
\end{equation*}
We note $n$ cannot be a square, so we have do not have any reduced forms $a(x^2+y^2)$.

In the second sum, we have either all three of $(r,s,t)$ are odd or one is even and two are odd.  We consider the reduced form for $r>s>t$:
\begin{equation*}
(s+t)x^2+2txy+(r+t)y^2. 
\end{equation*}
If all three $(r,s,t)$ are odd, we see $2\mid \gcd(s+t,2t,r+t)$ so the form is imprimitive.  Moreover, the form is counted with weight $(-1)$.  If in $(r,s,t)$ one is even and two are odd, we have three cases to check.  For example if $2r+1>2s+1>2t>0$, then the reduced form is
\begin{equation*}
(2t+2s+1)x^2+(4t)xy+(2t+2r+1),
\end{equation*}
which is counted with weight $(+1)$.  If the form is imprimitive, we see 
\begin{equation}
2\nmid \gcd(2t+2s+1,4t, 4t+2t+tr+1)>1.
\end{equation}
In all three cases the reduced forms of type (II) could be primitive or imprimitive. However, imprimitive forms with even $\gcd$ are counted with weight $(-1)$, the imprimitive forms with odd $\gcd>1$ are counted with weight $(+1)$, and the primitive forms are counted with weight $(+1)$.  What to do with this information will be clear at the end.

In the third sum, we have either all three of $(r,s,t)$ are odd or one is even and two are odd.  Suppose we have $r$ is odd and $t$ is even.  Rewriting to reflect parity, we have
{\allowdisplaybreaks \begin{align*}
r^2+2rt\rightarrow (4r+1)^2+2(4r+1)2t&=16r^2+8r+1+16rt+4t\not = 8m+3,\\
r^2+2rt\rightarrow (4r+3)^2+2(4r+3)2t&=16r^2+24r+9+16rt+12t\not = 8m+3,
\end{align*}}%
so we can only have all three odd.  Such solutions correspond to imprimitive reduced forms of type (III).  If $r>t$, we have the corresponding form $(r+t)x^2+2txy+(r+t)y^2.$ Because $r$ and $t$ are both odd, we see $2 \mid \gcd(r+t,2t,r+t)$.  If $t>r$, we have the corresponding form  $(2r)x^2+2rxy+(r+t)y^2.$  Because $r$ and $t$ are both odd, we see $2 \mid \gcd(2r,2r,r+t)$, so there are no primitive forms of type (III). Hence the third sum counts the imprimitive reduced forms of type (III) with weight $(-1)$: 
\begin{equation*}
 -\sum_{\substack{r,t\ge1\\r\ne t \\r^2+2rt=n}}1.
\end{equation*}

In the fourth sum, we can only have all three of $(r,s,t)$ are odd.  Hence the weight is $(-1)$.  Solutions of the form $3r^2=n$ correspond to the imprimitive reduced binary quadratic forms $2r(x^2+xy+y^2)$,  which contribute with weight $1/3$ to the Hurwitz class number.  Hence the fourth sum counts
\begin{equation*}
 -\frac{1}{3}\sum_{\substack{r\ge1\\3r^2=n}}1.
\end{equation*}

Assembling the above pieces, we have
\begin{equation*}
r_3(n)=12\cdot \Big \{ h(-4n) + \Big [ h(-4n)-H(4n)\Big ]+2 \Big [ A(-4n)\Big ]\Big \},
\end{equation*}
where $A(-4n)$ is the number of imprimitive reduced forms of discriminant $-4n$  with odd $\gcd>1$ from
\begin{equation*}
\frac{1}{2}\sum_{\substack{d\ge 1 \\ d \mid n}}1 \ \ \textup{and} \ \ 2\sum_{\substack{r>s>t>0\\rs+rt+st=n}}(-1)^{r+s+t}.
\end{equation*}
Thus we can write
{\allowdisplaybreaks \begin{align*}
r_3(n)&=12\cdot \Big \{ h(-4n) + \Big [ h(-4n)-H(4n)\Big ]+2 \Big [ \sum_{\substack{2\nmid d , d^2 \mid n}}h^{\prime}\Big (-\frac{4n}{d^2}\Big )-h(-4n)\Big ]\Big \}\\
&=12\cdot \Big \{ 2h(-4n) -H(4n) +2 \Big [ \sum_{\substack{ d^2 \mid 4n}}h^{\prime}\Big (-\frac{4n}{d^2}\Big )-\sum_{\substack{ (2d)^2 \mid 4n}}h^{\prime}\Big (-\frac{4n}{(2d)^2}\Big ) -h(-4n)\Big ]\Big \}\\
&=12\cdot \Big \{ 2h(-4n)-H(4n) +2 \Big [H(4n)-H(n) -h(-4n)\Big ]\Big \}\\
&=12\cdot \Big \{  H(4n)-2H(n) \Big \}\\
&=24H(n), 
\end{align*}}%
where we have used (\ref{equation:delta-3mod8}).

We consider the case $n\equiv 7 \pmod 8$.  Although this case is true by simpler means, it is nice to see that our method gives the correct answer.   The argument is almost exactly the same as that for the previous case; however, at the very end when we obtain
\begin{align*}
r_3(n)&=12\cdot \Big \{  H(4n)-2H(n) \Big \},
\end{align*}
we instead use (\ref{equation:delta-7mod8}), which gives
\begin{equation*}
r_3(n)=0.
\end{equation*}

\section{Concluding Remarks}\label{section:remarks}

Authors such as Bell \cite{Be}, Borwein and Choi \cite{BC}, Gage \cite{Ga}, Lass \cite{La}, Mordell \cite{Mo0, Mo1}, and Peters \cite{Pe} have studied solutions to the equation 
\begin{equation}
rs+rt+st=n.\label{equation:rstn}
\end{equation}
It turns out that our proof of Theorem \ref{theorem:gauss-withsquare} is not unlike the methods used by Mordell \cite{Mo0, Mo1}, where he counts non-negative solutions to (\ref{equation:rstn}) by establishing a correspondence with reduced binary quadratic forms with discriminant $\Delta=-4n$.  However, our implementation is different than what one finds in \cite{Mo0, Mo1}.  For example, we have to work with a weight term of $(-1)^{r+s+t}$, and our motivation comes from investigating limiting cases  of higher-dimensional Kronecker-type identites.  In counting non-negative solutions to (\ref{equation:rstn}), Mordell adopts the usual convention of assigning weight $1/3$ to corresponding forms of type $2a(x^2+xy+y^2)$ and weight $1/2$ to forms of type $a(x^2+y^2)$ but also assigns $1/2$ to corresponding forms of type $ax^2+cy^2$ where $a\ne c$, which seems unjustified.  Our approach makes clear the associated weights of $1/2$ and $1/3$, see for example identity (\ref{equation:AC-qf-weights}).   Mordell does not make explicit the connection with Theorem \ref{theorem:gauss-withsquare}; however, Lass \cite{La} does relate solutions of (\ref{equation:rstn}) to sums of three squares but in a different sense.

\section*{Acknowledgements}
We would like to thank Neil Sloane's On-line Encyclopedia of Integer Sequences for directing us to references \cite{BC, Cr, Pe, We},  Michel Matignon for pointing out \cite{La}, Peter Sarnak for pointing out \cite{Weil2}, and Bob Vaughan for pointing out \cite{UH}.   We would also like to thank Hjalmar Rosengren and Ole Warnaar for helpful comments and suggestions.

\end{document}